\documentclass[12pt]{article}
\usepackage{amsmath,amssymb, amsthm}
\usepackage{eso-pic}
\usepackage{makeidx}
\usepackage{subfiles}
\usepackage{hyperref}
\usepackage[capitalise]{cleveref}
\usepackage{hyperref}
\usepackage{eucal}
\usepackage[T1]{fontenc}

\makeatletter
\def\printnotation{{%
\def\indexname{Index of notation}
\begin{theindex}
\@input{\jobname.ntn}
\end{theindex}
}}
\makeatother

\makeglossary
\usepackage{amsfonts,amssymb,amscd,amsmath,latexsym,amsbsy, setspace}
\usepackage{amsfonts,amssymb,graphicx,amsthm,mathtools}
\usepackage{mathrsfs}
\usepackage{subfiles}
\usepackage[all,cmtip,curve, knot,frame]{xy} 
\usepackage{amssymb}
\usepackage{amsfonts}
\usepackage{latexsym}
\usepackage{epstopdf}
\usepackage{todonotes}
\usepackage{tikz}
\usepgflibrary{shapes.geometric}
\usetikzlibrary{matrix, arrows.meta}

\usepackage{enumitem}
\setlist{noitemsep}

\newcommand{\arxiv}[1]{\href{http://arxiv.org/abs/#1}{\tt arXiv:\nolinkurl{#1}}}

\newtheorem{theorem}{Theorem}[section]
\newtheorem{lemma}[theorem]{Lemma}

\newtheorem{conjecture}[theorem]{Conjecture}
\newtheorem{proposition}[theorem]{Proposition}

\theoremstyle{definition}
\newtheorem{definition-proposition}[theorem]{Definition-Proposition}
\newtheorem{definition}[theorem]{Definition}
\newtheorem{notation}[theorem]{Notation}
\newtheorem{example}[theorem]{Example}
\newtheorem{remark}[theorem]{Remark}

\newtheorem{question}[theorem]{Question}

\newcommand{\LFP}{\operatorname{\mathbf{Pr}}}

\newcommand{\mop}{\otimes\mathrm{op}}
\newcommand{\bop}{\sigma\mathrm{op}}
\newcommand{\op}{\mathrm{op}}

\newcommand{\Vect}{\operatorname{Vect}}

\newcommand{\Dr}{\operatorname{Dr}}
\newcommand{\End}{\operatorname{End}}
\newcommand{\forget}{\operatorname{forget}}
\newcommand{\free}{\operatorname{free}}
\newcommand{\HC}{\operatorname{HC}}
\newcommand{\Hom}{\operatorname{Hom}}

\newcommand{\Ind}{\text{Ind}\,}
\newcommand{\LMod}{\operatorname{LMod}}
\newcommand{\Mod}{\operatorname{Mod}}

\newcommand{\Rep}{\operatorname{Rep\,}}

\newcommand{\triv}{\mathrm{triv}}

\newcommand{\tens}{\mathrm{tens}}
\newcommand{\U}{\mathrm{U}}

\newcommand{\cA}{\mathcal{A}}

\newcommand{\cB}{\mathcal{B}}
\newcommand{\cC}{\mathcal{C}}

\newcommand{\cD}{\mathcal{D}}
\newcommand{\cF}{\mathcal{F}}

\newcommand{\cM}{\mathcal{M}}

\newcommand{\cT}{\mathcal{T}}
\newcommand{\cN}{\mathcal{N}}
\newcommand{\cS}{\mathcal{S}}

\newcommand{\QCoh}{\operatorname{QCoh}}

\newcommand{\Fun}{\operatorname{\operatorname{Fun}}}

\newcommand{\cE}{\mathcal{E}}

\newcommand{\bt}{\boxtimes}
\newcommand{\id}{\operatorname{id}}
\newcommand{\Emb}{\operatorname{Emb}}
\newcommand{\RR}{\mathbb{R}}

\newcommand{\rZ}{\mathrm{Z}}

\newcommand{\coev}{\operatorname{coev}}
\newcommand{\ev}{\operatorname{ev}}

\newcommand{\Tens}{\operatorname{\mathbf{Tens}}}
\newcommand{\BrTens}{\operatorname{\mathbf{BrTens}}}

\newcommand{\BrFus}{\operatorname{\mathbf{BrFus}}}
\newcommand{\Fus}{\operatorname{\mathbf{Fus}}}

\newcommand{\alg}{\operatorname{Alg}}

\newcommand{\un}{\mathbf{1}}

\newcommand{\Disk}{\operatorname{Disk}}	
\newcommand{\Mfld}{\operatorname{Mfld}}
\newcommand{\Pic}{\operatorname{Pic}}
\newcommand{\hPic}{\underline{\operatorname{Pic}}}
\usepackage{soul}

\usepackage{textcomp}

\newcommand{\defterm}[1]{\textbf{\emph{#1}}}

\def\HH{\hbox{${\mathcal H}$\kern-5.2pt${\mathcal H}$}}

\usepackage{fullpage}

\begin{document}
\title{Invertible braided tensor categories}
\author{Adrien Brochier, David Jordan, Pavel Safronov, Noah Snyder}

\maketitle
\begin{abstract}
We prove that a finite braided tensor category $\cA$ is invertible in the Morita $4$-category $\BrTens$ of braided tensor categories if, and only if, it is non-degenerate.  This includes the case of semisimple modular tensor categories, but also non-semisimple examples such as categories of representations of the small quantum group at good roots of unity.  Via the cobordism hypothesis, we obtain new invertible $4$-dimensional framed topological field theories, which we regard as a non-semisimple framed version of the Crane-Yetter-Kauffman invariants, after Freed--Teleman and Walker's construction in the semisimple case.  More generally, we characterize invertibility for $E_1$- and $E_2$-algebras in an arbitrary symmetric monoidal $\infty$-category, and we conjecture a similar characterization of invertible $E_n$-algebras for any $n$.  Finally, we propose the Picard group of $\BrTens$ as a generalization of the Witt group of non-degenerate braided fusion categories, and pose a number of open questions about it.
\end{abstract}
\setcounter{tocdepth}{2}
\tableofcontents

\section{Introduction}

In the paper \cite{Brochier2018}, we introduced a symmetric monoidal 4-category $\BrTens$ whose objects are braided tensor categories, and whose morphisms encode their higher Morita theory, following \cite{Haugseng2017, Johnson-Freyd2017}, and gave sufficient conditions for 3-dualizability (``cp-rigidity'') and 4-dualizability (fusion) in $\BrTens$.  In this paper we consider the related question of invertibility in $\BrTens$.  We also treat both dualizability and invertibility in the more general setting of $E_2$-algebras in an arbitrary background symmetric monoidal 2-category $\cS$.

\subsection{Main results}

Finite braided tensor categories are linear and abelian braided monoidal categories satisfying strong finiteness and rigidity conditions (see \cref{sec:tensor} for more details).  Such an $\cA$ is called \defterm{non-degenerate} if the M\"uger center of $\cA$ is trivial, i.e. if for every non-trivial object $X\in\cA$, there exists an object $Y\in\cA$ such that the double-braiding $\sigma_{Y,X}\circ \sigma_{X,Y}$ on $X\otimes Y$ is not the identity.  The main result of this paper is:

\begin{theorem}\label{thm:intro-main-result}
A finite braided tensor category $\cA$ is an invertible object of $\BrTens$ if, and only if, $\cA$ is non-degenerate.
\end{theorem}

We note that \cref{thm:intro-main-result} includes modular tensor categories, whose invertibility is known to experts through unpublished theorems of Freed--Teleman and Walker (c.f. \cite{Freed2012a,Freed2012b} and \cite{Walker}), but we emphasize that we require neither semisimplicity nor a ribbon structure.  In particular, our results include representation categories of small quantum groups for good roots of unity, which are non-degenerate but not semisimple.

\cref{thm:intro-main-result} has an application to the construction of topological field theories via the cobordism hypothesis \cite{Baez1995, Lurie2009, MR2742424, MR2994995, Ayala2017a}.  Namely, we obtain a $4$-dimensional fully extended framed TFT attached to every non-degenerate finite braided tensor category, which we may regard as a framed and non-semisimple analog of the fully extended Crane-Yetter-Kauffman \cite{Crane1993,Crane1997,Walker2011, Barenz2018,Kirillov2020} topological field theory envisioned by Freed--Teleman and Walker.  These TFTs are invertible in the sense of \cite{MR3330283, SchommerPriesInvertible, debray2019lectures}.  We believe that this aspect of the work will be important for applications, as it was not generally expected that \emph{non-semisimple} braided tensor categories would give rise to $4$-dimensional TFTs due to the heavy reliance on semisimplicity in the traditional state-sum approach to Crane-Yetter-Kauffman TFT's. In future work we plan to study $SO(4)$-fixed point structures (hence the associated oriented TFT's), and to give a reformulation of non-semisimple Witten-Reshetikhin-Turaev theories of \cite{Costantino2014, renzi20193dimensional} as field theories relative to our non-semisimple Crane-Yetter-Kauffman theory, following the proposal of Freed--Teleman and Walker.

\cref{thm:intro-main-result} also gives rise to a generalization of the Witt group of non-degenerate braided fusion categories from \cite{MR3039775}, via the Picard group of $\BrTens$.  In \cref{sec:Witt}, we discuss this Picard group, its relation to the Witt group, and a number of natural open questions related to it.

Braided tensor categories define $E_2$-algebras in the symmetric monoidal $2$-category $\cS=\LFP$ of locally presentable linear categories.  Our approach in this paper is to work as much as possible in the more general $(\infty,4)$-category $\alg_2(\cS)$ of $E_2$-algebras in an arbitrary ambient closed symmetric monoidal $(\infty,2)$-category $\cS$.

In this generality we have analogs $\rZ_0(\cA)$, $\rZ_1(\cA)$, $\rZ_2(\cA)$, respectively, of the endo-functors, the Drinfeld center, and the M\"uger center of $\cA$ (see \cref{sect:Enalgebras} for detailed definitions).  We will also use the Harish-Chandra category $\HC(\cA)$ (a variant of the monoidal Hochschild homology with the annulus/bounding framing rather than cylinder/product framing), and we will denote by $\cA^{\mop}$ and $\cA^{\bop}$ the $E_2$-algebras obtained as reflections of the $E_2$-structure through the $x$- and $y$-axes (see \cref{sect:Enalgebras}). We first prove the following result:

\begin{theorem} \label{thm:Characterize3Dualizability} An $E_2$-algebra $\cA\in\alg_2(\cS)$ is 3-dualizable if, and only if, $\cA$ is dualizable as an object of $\cS$, as an $\cA^e$-module, and as an $\HC(\cA)$-module. 
\end{theorem}

Turning next to invertibility we prove:

\begin{theorem}\label{thm:intro-main-result-E1}
A 2-dualizable $E_1$-algebra $\cC\in\alg_1(\cS)$ is invertible if, and only if:
\begin{enumerate}
\item It is \defterm{central}: the natural morphism $\un_\cS\to\rZ_1(\cC)$ is an equivalence.
\item It is \defterm{Azumaya}: the natural morphism $\cC\bt \cC^{\mop}\rightarrow \End(\cC)$ is an equivalence.
\end{enumerate}
\end{theorem}

In the case $\cS$ is the category of $R$-modules for $R$ a commutative ring, \cref{thm:intro-main-result-E1} reduces to a characterization of Azumaya algebras, see \cref{ex:Azumaya}.  Our main general result is the following characterization of invertible $E_2$-algebras:

\begin{theorem}\label{thm:intro-main-result-general} A 3-dualizable $E_2$-algebra $\cA\in \alg_2(\cS)$ is invertible if, and only if:
\begin{enumerate}
\item It is \defterm{non-degenerate}: the natural morphism $\un_\cS\to \rZ_2(\cA)$ is an equivalence.
\item It is \defterm{factorizable}: the natural morphism $\cA\bt \cA^{\bop}\to \rZ_1(\cA)$
is an equivalence.
\item It is \defterm{cofactorizable}:  the natural morphism $\HC(\cA) \to \End(\cA)$ is an equivalence.
\end{enumerate}
\end{theorem}

Our proof of \cref{thm:intro-main-result} is an application of \cref{thm:intro-main-result-general} in the case $\BrTens=\alg_2(\LFP)$.  For this we require two additional claims: that finite braided tensor categories are $3$-dualizable, and that for finite braided tensor categories, non-degeneracy, factorizability and cofactorizability are mutually equivalent.  The first claim follows from the main result of \cite{Brochier2018} since finite tensor categories are cp-rigid.  For the second claim, we have the following two theorems, the second of which we prove in \cref{sec:tensor} (see \cref{thm:body-main-result} for the complete statement).

\begin{theorem}[\cite{ShimizuND}]\label{thm:Shimizu} A finite braided tensor category is factorizable if and only if it is non-degenerate.\end{theorem}

\begin{theorem} A finite braided tensor category is factorizable if, and only if, it is co-factorizable.\end{theorem}

\begin{remark}
Outside of finite braided tensor categories, non-degeneracy does not imply factorizability, nor co-factorizability.  For example, the category of integrable representations of a semisimple quantum group at generic quantization parameter is non-degenerate, but is clearly not factorizable.  We do not know whether either factorizability or co-factorizability imply finiteness (see \cref{q:invimplfin}).
\end{remark}

Besides clarifying proofs, the generality of \cref{thm:intro-main-result-general} makes it possible to consider invertibility of $E_2$-algebras in contexts other than finite braided tensor categories.  For instance, the 3-dimensional Rozansky--Witten TFT \cite{RozanskyWitten} associated to a holomorphic symplectic manifold $X$ gives rise to a ribbon braided tensor structure on the bounded derived category of coherent sheaves $D^b_{coh}(X)$ \cite{RobertsWillerton}.  We expect that this braided tensor category is invertible in the suitably derived version of $\BrTens$, and hence defines an invertible 4D topological field theory. In addition, we expect that one may formulate the Rozansky--Witten TFT as a field theory relative to the resulting 4D theory, in precisely the way that the Witten-Reshetikhin--Turaev 3D TFT \cite{Witten1989, Reshetikhin1991} is constructed relative to the 4D Crane--Yetter--Kauffman theory.  We hope to return to this in future work.

\subsection{Conjectural extension to $E_n$-algebras}

Let us now collect two conjectures generalizing our results for dualizability and invertibility of $E_1$ and $E_2$-algebras to $E_n$-algebras.  An important theorem of \cite{Gwilliam2018} (c.f. \cite[Claim 4.1.14]{Lurie2009}) states that every $E_n$ algebra $\cA\in\alg_n(\cS)$ is $n$-dualizable.  The following characterization for $(n+1)$-dualizability of $E_n$-algebras was formulated in \cite[Remark 4.1.27]{Lurie2009}, where it is remarked that it would follow from an unpacking of the proof of the cobordism hypothesis.  It is formulated in terms of the factorization homology (or topological chiral homology) $\int_M\cA$ of a framed $n$-manifold $M$, with coefficients in an an $E_n$-algebra $\cA$ (see \cref{def:facthom} for a brief recollection, or \cite{Ayala2015,Lurie2009} for complete definitions).

Let us fix the framing on $S^{k-1}\times \RR$ which bounds an $k$-disk, and its product framing on $S^{k-1}\times \RR^{n-k+1}$ with the trivial framing on $\RR^{n-k}$. This induces an action of the $E_{n-k+1}$-algebra $\int_{S^{k-1}\times \RR^{n-k+1}} \cA$ on $\int_{\RR^n} \cA\cong \cA$.
\begin{conjecture}
\label{conj:Lurie-dualizability}
An $E_n$-algebra $\cA\in\alg_n(\cS)$ is $(n+1)$-dualizable if, and only if, it is dualizable over the factorization homologies $\int_{S^{k-1}\times \RR^{n-k+1}} \cA$ for $k=0,\ldots, n$.
\end{conjecture}
We note that the forward implication is clear.  The case $n=1$ of the conjecture is proved in \cite{Lurie2009}, and recalled in \cref{thm:2dualizabilityE1}. We prove the case $n=2$ relevant to the present paper in \cref{thm:3dualizabilityE2}.

Recall that the $E_k$-center $\rZ_k(\cA)$ of an $E_n$-algebra $\cA$, for $k\leq n$, is:
\[\rZ_k(\cA):= \End_{\int_{S^{k-1}\times \RR^{n-k+1}} \cA}(\cA).\]
The pattern of \cref{thm:intro-main-result-general} leads us to make the following analog of \cref{conj:Lurie-dualizability}:
\begin{conjecture}\label{conj:invertibility}
An $E_n$-algebra $\cA$ is invertible if, and only if, it is $(n+1)$-dualizable and the canonical maps
\[\int_{S^{k-1}\times \RR^{n-k+1}} \cA\to \rZ_{n-k}(\cA)\]
are equivalences for $k = 0,\ldots, n$.
\end{conjecture}

Again the forward implication is clear.  \cref{thm:E1invertible,thm:E2invertible} confirm the conjecture in the cases $n=1$ and $n=2$, and the same techniques should work in the general case provided one has a sufficiently general calculus of mates (see \cref{rem:mates}).  We therefore expect \cref{conj:invertibility} to be significantly easier to prove than \cref{conj:Lurie-dualizability}. In particular, in contrast to \cref{conj:Lurie-dualizability}, we expect it can be proved independently from the proof or even the statement of the cobordism hypothesis.    In order to apply this conjecture to any examples, however, one must verify $(n+1)$-dualizability, which involves the previous more difficult conjecture.

\subsection{Acknowledgements}
DJ is supported by European Research Council (ERC) under the European Union's Horizon 2020 research and innovation programme (grant agreement no. 637618). PS is supported by the NCCR SwissMAP grant of the Swiss National Science Foundation. NS is supported by NSF grant DMS-1454767. DJ and NS were partially supported by NSF grant DMS-0932078, administered by the Mathematical Sciences Research Institute while they were in residence at MSRI during the Quantum Symmetries program in Spring 2020.  The authors are grateful also to the ICMS RIGS program for hospitality during a research visit to Edinburgh.

We would like to thank Dan Freed and Dmitri Nikshych helpful questions and comments concerning the Picard group of $\BrTens$. NS would like to thank Sean Sanford for excellent expositional talks on Shimizu's work. 
\section{Dualizability and invertibility of $E_1$- and $E_2$-algebras}

In this section we recall some fundamental definitions and results about the Morita category of $E_n$-algebras, and we prove our main general results about dualizability and invertibility for $E_2$-algebras.  To clarify notation, we will keep a running example $\cS=\LFP$, but we stress that the results in this section are all general, and that any results which are specific to braided tensor categories are delayed until the next section.

\subsection{$E_n$-algebras and higher Morita categories}
\label{sect:Enalgebras}
	
We begin by briefly recalling the notions of an $E_n$-algebra in a closed symmetric monoidal $(\infty,2)$-category $\cS$ which admits geometric realizations.  We denote its symmetric monoidal structure by $\boxtimes$.  We refer to \cite[Chapter 5]{Lurie} and \cite{Ayala2015} for more details on the following definitions.

\begin{definition}
Let $\Mfld_n^{fr}$ denote the topological category whose objects are framed $n$-manifolds, and whose space of morphisms between $M$ and $N$ is the topological space $\Emb^{fr}(M,N)$ of framed embeddings\footnote{Recall that a framed embedding of $M$ into $N$ is an embedding of topological manifolds, together with an isotopy $\gamma$ between the framing on $M$ and the framing pulled back from $N$.  In particular, a framed embedding is not required to preserve the framing strictly.}.  We equip $\Mfld_n^{fr}$ with the structure of a symmetric monoidal category under disjoint union.  We denote by $\Disk_n^{fr}$ the full subcategory consisting of finite disjoint unions of the standard open disk $(0,1)^n$
\end{definition}

\begin{definition}
\label{def:facthom}
An \defterm{$E_n$-algebra} in $\cS$ a symmetric monoidal functor $\cA\colon\Disk_n^{fr}\to\cS$.  The \defterm{factorization homology} with coefficients in $\cA$ is the left Kan extension of the functor $\cA$ along the inclusion of $\Disk_n^{fr}$ into $\Mfld_n^{fr}$, and is denoted by $M\mapsto\int_M\cA$.
\end{definition}

\begin{example}\label{ex:tens-cat}
In the familiar example $\cS=\LFP$ (see \cref{sec:tensor}), an $E_1$-algebra is a locally presentable category tensor category with a colimit preserving tensor product.  Henceforth we will call this a `tensor category', without a further requirement of rigidity, and will make explicit any notion of rigidity (e.g., compact-rigid or cp-rigid) we assume.  Similarly, an $E_2$-algebra in $\LFP$ is a braided tensor category, and an $E_k$-algebra, for $k\geq 3$, is a symmetric tensor category.

When we identify $E_2$-algebras in $\LFP$ with braided tensor categories, we use the following conventions, following \cite{Brochier2018}: the tensor multiplication is in the $x$-direction, and the braiding is given by counterclockwise rotation.  
\end{example}

Let us fix some more notation for later use. An $E_1$-algebra $\cC$ has an opposite $E_1$-algebra $\cC^{\mop}$ coming from precomposing by reflection.  The notation $\mop$ is intended to emphasize that we are taking the dual in the multiplication direction, and not taking a dual in the underlying category $\cS$.  An $E_2$-algebra $\cA$ has two {\it a priori} distinct notions of opposite:  $\cA^{\mop}$ where we reflect in the $x$-direction and $\cA^{\bop}$ where we reflect in the $y$-direction.   We have two canonical equivalences of $E_2$-algebras
\begin{equation}
\cA^{\mop}\simeq\cA^{\bop},
\label{eq:mopbop}
\end{equation}
given by a 180 degree rotation in the clockwise and counter-clockwise directions.  

\begin{example}
Consider the case $\cS=\LFP$ and suppose $\cA\in\alg_2(\cS)$ is a braided tensor category. The braided tensor category $\cA^{\mop}$ has the same underlying category, the tensor product $x\otimes^{\op} y = y\otimes x$ and the braiding
\[x\otimes^{\op} y = y\otimes x\xrightarrow{\sigma^{-1}_{x, y}} x\otimes y = y\otimes^{\op} x.\]
The braided tensor category $\cA^{\bop}$ has the same underlying tensor category and the braiding $\sigma^{-1}_{y, x}\colon x\otimes y\rightarrow y\otimes x$. The two canonical equivalences in this case are each given by equipping the identity functor $\id\colon \cA^{\mop}\rightarrow \cA^{\bop}$ with a braided tensor structure given by the braiding, and its inverse.
\end{example}

\begin{definition}[{See \cite[Section 7]{Ginot2015}}]
Let $\cA$ be an $E_n$-algebra. Its \defterm{enveloping algebra} is the $E_1$-algebra
\[\U_\cA^{E_n} = \int_{S^{n-1}\times \RR} \cA,\]
where $S^{n-1}\times \RR$ carries the framing bounding the $n$-ball.
\end{definition}

The enveloping algebra $\U_\cA^{E_n}$ has a natural left module action on $\int_{\RR^n} \cA\cong \cA$, coming from the sphere bounding the ball.

\begin{definition}
Let $\cA$ be an $E_n$-algebra. Its \defterm{$E_n$-center} is the object
\[\rZ_n(\cA) = \End_{\U_\cA^{E_n}}(\cA)\in\cS.\]
\end{definition}

By \cite[Proposition 3.16]{Francis2013} we have an equivalence of $\infty$-categories
\[\LMod_{\U_\cA^{E_n}}\cong \Mod_\cA^{E_n}\]
between the $\infty$-category of left $\U_\cA^{E_n}$-modules and the $\infty$-category of $E_n$-$\cA$-modules. Moreover, the notion of a center introduced in \cite[Definition 5.3.1.6]{Lurie} is shown in \cite[Theorem 5.3.1.30]{Lurie} to coincide with $\End_{\Mod_\cA^{E_n}}(\cA)$. So, $\rZ_n(\cA)$ is indeed the center of the $E_n$-algebra in the sense of Lurie.

In addition, by \cite[Theorem 3.3.3.9]{Lurie} $\Mod_\cA^{E_n}$ is an $E_n$-monoidal $\infty$-category with $\cA\in\Mod_\cA^{E_n}$ the unit object. So, by the Dunn--Lurie additivity theorem \cite[Theorem 5.1.2.2]{Lurie} (cf. \cite{Dunn, MR3311768}) $\rZ_n(\cA)$ is an $E_{n+1}$-algebra.

\begin{example}
Suppose $\cC$ is an $E_1$-algebra. We denote
\[\cC^e = \U_\cC^{E_1}\cong \cC\bt\cC^{\mop}.\]
In the case $\cS=\LFP$, it is shown in \cref{prop:Drinfeldcenter} that the $E_1$-center $\rZ_1(\cC)$ coincides with the Drinfeld center of the tensor category $\cC$.
\end{example}

\begin{example}
Suppose $\cA$ is an $E_2$-algebra. We denote
\[\HC(\cA) = \U_\cA^{E_2}\cong \cA\bt_{\cA\bt\cA^{\bop}} \cA^{\mop},\]
where we regard $\cA$ (resp., $\cA^{\mop}$) as an $E_1$-algebra in right (resp., left) $\cA\bt\cA^{\bop}$-modules, hence the relative tensor product inherits again an $E_1$-structure. In the case $\cS=\LFP$, it is shown in \cref{prop:Mugercenter} that the $E_2$-center $\rZ_2(\cA)$ of coincides with the M\"uger center of the braided tensor category $\cA$.
\end{example}


The collection of $E_n$-algebras in a fixed $\cS$ carries the structure of a symmetric monoidal $(\infty, n+2)$-category $\alg_n(\cS)$.  We refer to the papers \cite{Johnson-Freyd2017,Haugseng2017,Scheimbauer2014,Calaque} for a rigorous construction of $\alg_n(\cS)$ in the model of iterated complete Segal spaces, and to \cite{Gwilliam2018} for a digestable exposition, and a treatment of dualizability.  Let us remark for experts that we will require Haugseng's ``unpointed'' model in order to treat questions of higher dualizability and invertibility (see \cite[\S 1.4]{Gwilliam2018} and the discussion there).

While the works cited above are required in order to make rigorous sense of the composition laws of higher morphisms, and their compatibilities, it is possible nevertheless to give an informal description of objects and morphisms themselves, in $\alg_1(\cS)$ and $\alg_2(\cS)$:

\begin{definition}[Sketch] The $(\infty,3)$-category $\alg_1(\cS)$ has:
\begin{itemize}
\item As objects $E_1$-algebras $\cA,\cB$, \ldots
\item As 1-morphisms the $(\cA,\cB)$-bimodules $\cM,\cN,\ldots$
\item As 2-morphisms the bimodule 1-morphisms
\item As 3-morphisms the bimodule 2-morphisms
\end{itemize}
The symmetric monoidal structure, as well as the composition of 2- and 3-morphisms, are those inherited from $\cS$.  Composition of $1$-morphisms is the relative tensor product of bimodules.
\end{definition}

\begin{definition}[Sketch] The $(\infty,4)$-category $\alg_2(\cS)$ has:
\begin{itemize}
\item As objects $E_2$-algebras $\cA,\cB$, \ldots
\item As 1-morphisms the $E_1$-algebra objects $\cC$, $\cD$, \ldots in $(\cA,\cB)$-bimodules,
\item As 2-morphisms the $(\cC,\cD)$-bimodule objects $\cM,\cN,\ldots$ in $(\cA,\cB)$-bimodules
\item As 3-morphisms the bimodule 1-morphisms
\item As 4-morphisms the bimodule 2-morphisms
\end{itemize}
The symmetric monoidal structure, as well as the composition of 3- and 4-morphisms, are compatible with those in $\cS$.  For example, to compose $3$-morphisms we endow the composition of the underlying $1$-morphisms with the structure of a bimodule $1$-morphism in an appropriate way. Composition of $1$- and $2$-morphisms is given by the relative tensor product of bimodules, equipping the resulting composition in the case of $1$-morphisms with a canonical $E_1$-algebra structure.
\end{definition}

There is a potential ambiguity in the notion of $(\cA,\cB)$-bimodule appearing above, since $\cA$ and $\cB$ each admit multiplications in both the $x$- and $y$- directions.  Hence let us fix the following conventions, following \cite{Brochier2018}:

\begin{itemize}
\item The $(\cA,\cB)$-bimodule structure on a $1$-morphism is with respect to multiplication \emph{in the $y$-direction} for both $\cA$ and $\cB$, whereas (by \cref{ex:tens-cat}) the underlying $E_1$-algebra associated to an $E_2$-algebra is in the $x$-direction. This means that an $(\cA,\cB)$-bimodule is an $\cA\bt\cB$-module (rather than an $\cA\bt\cB^{op}$-module).

\item Hence the data of an $E_1$-algebra in $(\cA,\cB)$-bimodules is equivalent to the data of an $E_1$-algebra $\cC$ in $\cS$, together with a morphism $\cA\bt\cB^{\bop}\to \rZ_1(\cC)$.  We call such data a \defterm{$(\cA,\cB)$-central algebra}.

\item Given $E_2$-algebras $\cA, \cB$ and $(\cA, \cB)$-central algebras $\cC, \cD$ a $(\cC, \cD)$-bimodule object $\cM$ in the monoidal category of $(\cA, \cB)$-bimodules is an \defterm{$(\cA, \cB)$-centered $(\cC, \cD)$-bimodule}.
\end{itemize}

\subsection{Dualizability for $E_1$- and $E_2$-algebras}
Recall from \cite[Definition 2.3.16]{Lurie2009} that a symmetric monoidal $(\infty,k)$-category \defterm{has duals} if every object has a dual, and every $i$-morphism has a left and right adjoint for $1 \leq i < k$.  An object in an $(\infty,n)$-category is called \defterm{$k$-dualizable} if it belongs to a full sub $(\infty, k)$-category which has duals.

In this section we discuss dualizability and adjointability in the Morita categories $\alg_1(\cS)$ and $\alg_2(\cS)$.  The case of $\alg_1(\cS)$ is well-known, but the results for $\alg_2(\cS)$ are new.  We begin by looking at dualizability in $\alg_1(\cS)$.

\begin{proposition}
\label{thm:1dualizabilityE1}
Every $E_1$-algebra $\cC\in\alg_1(\cS)$ is 1-dualizable with dual $\cC^\vee=\cC^{\mop}$, and with evaluation and co-evaluation given by versions of the regular bimodule:
\begin{itemize}
\item $\ev$ is $\cC$ as a $(\cC^{\mop}\bt\cC,\un_\cS)$-bimodule.
\item $\coev$ is $\cC$ as a $(\un_\cS,\cC\bt\cC^{\mop})$-bimodule.
\end{itemize}
\end{proposition}

It turns out that a bimodule having a left adjoint or a right adjoint depends only on the left action or the right action respectively, as explained in the following definition and proposition.

\begin{definition}
Let $\cC$ be an $E_1$-algebra in $\cS$.  A right (resp. left) $\cC$-module $\cM$ is called \defterm{dualizable} if it has a right (resp. left) adjoint as a $(\cC,\un_\cS)$ bimodule (resp. $(\un_\cS,\cC)$ bimodule).
\end{definition}

\begin{proposition}[{\cite[Proposition 4.6.2.13]{Lurie}}]\label{prop:E1Hom1adjoints} A 1-morphism $M\colon\cC\to\cD$ in $\alg_1(\cS)$ has a right (resp. left) adjoint if, and only if, $M$ is dualizable as a $\cD$-module (resp. $\cC$-module).
\end{proposition}

In order to give a complete characterization of 2-dualizable objects in $\alg_1(\cS)$, we first recall the following result (see \cite[Proposition 4.2.3]{Lurie2009} and \cite[Theorem 3.9]{Pstragowski2014}) which reduces 2-dualizability to a finite number of conditions.

\begin{theorem}\label{thm:easy2dualizability}
A 1-dualizable object $X$ of a symmetric monoidal 2-category is $2$-dualizable if, and only if, the evaluation and coevaluation maps each admit a right adjoint.
\end{theorem}

We may finally state the following well-known characterization of 2-dualizable objects in $\alg_1(\cS)$.

\begin{theorem}
An $E_1$-algebra $\cC\in\alg_1(\cS)$ is 2-dualizable if, and only if, it is dualizable as an object of $\cS$, and as a $\cC^e$-module.
\label{thm:2dualizabilityE1}
\end{theorem}
\begin{proof}
We recall from \cref{thm:1dualizabilityE1} the 1-dualizability data $\cC^\vee$, $\ev$, $\coev$.  By \cref{thm:easy2dualizability}, $\cC$ is 2-dualizable if, and only if, $\ev$ and $\coev$ both admit right adjoints. Then by \cref{prop:E1Hom1adjoints}, these right adjoints exist if, and only if, $\cC$ is dualizable as a $\un_\cS$-module (i.e. as an object of $\cS$), and as a $\cC^e$-module.
\end{proof}

Now we turn to 2- and 3-dualizability of $E_2$-algebras. In particular, we prove the $n=2$ case of  \cref{conj:Lurie-dualizability}).

\begin{lemma}
If $\cC$ is an $E_1$-algebra, there is a canonical equivalence of $E_2$-algebras
\begin{equation}
\rZ_1(\cC^{\mop}) \cong \rZ_1(\cC)^{\mop}.
\label{eq:mopDrinfeldCenter}
\end{equation}
\end{lemma}


\begin{theorem}[See {\cite[Section 4]{Gwilliam2018}}]
Let $\cC\colon \cA\rightarrow \cB$ be a 1-morphism in $\alg_2(\cS)$ given by an $E_1$-algebra $\cC$ equipped with an $(\cA, \cB)$-central structure $\cA\bt \cB^{\bop}\rightarrow \rZ_1(\cC)$. Its right adjoint is $\cC^{\mop}$ equipped with a $(\cB, \cA)$-central structure via the composite
\[\cB\bt \cA^{\bop}\longrightarrow \rZ_1(\cC)^{\bop}\cong \rZ_1(\cC)^{\mop}\cong \rZ_1(\cC^{\mop}),\]
where the penultimate equivalence is given by a $180$-degree clockwise rotation \eqref{eq:mopbop}, and the last equivalence is given by \eqref{eq:mopDrinfeldCenter}. The unit of the adjunction is $\cC$ viewed as $(\cA, \cA)$-centered $(\cA, \cC\bt_\cB \cC^{\mop})$-bimodule. The counit of the adjunction is $\cC$ viewed as a $(\cB, \cB)$-centered $(\cC^{\mop}\bt_\cA \cC, \cB)$-bimodule.
\label{thm:E2Hom1adjoints}
\end{theorem}

\begin{remark}
The right adjoint constructed above used clockwise rotation, the left adjoint would use counterclockwise rotation.
\end{remark}

\begin{theorem}[See {\cite[Section 4]{Gwilliam2018}}]
Every $E_2$-algebra $\cA\in\alg_2(\cS)$ is 2-dualizable with dual $\cA^\vee = \cA^{\bop}$, and with evaluation and coevalation given by the regular central algebra,
\begin{itemize}
\item $\ev$ is $\cA$ as a $(\cA^{\bop}\bt\cA,\un_\cS)$-central algebra,
\item $\coev$ is $\cA$ as a $(\un_\cS,\cA\bt\cA^{\bop})$-central algebra.
\end{itemize}
The right adjoints to evaluation and coevaluation are given by
\begin{itemize}
\item $\ev^R$ is $\cA^{\mop}$ as a $(\un_\cS,\cA^{\bop}\bt\cA)$-central algebra,
\item $\coev^R$ is $\cA^{\mop}$ as a $(\cA\bt\cA^{\bop},\un_\cS)$-central algebra,
\end{itemize}
as in \cref{thm:E2Hom1adjoints}. Their unit and counit morphisms are given by:
\begin{itemize}
\item $\eta_{\coev}$ is $\cA$ as a $(\un_\cS, \HC(\cA))$-bimodule.

\item $\epsilon_{\coev}$ is $\cA$ as a $(\cA\bt\cA^{\bop}, \cA\bt\cA^{\bop})$-centered $(\cA^{\mop}\bt \cA, \cA\bt \cA)$-bimodule.

\item $\eta_{\ev}$ is $\cA$ as a $(\cA^{\bop}\bt\cA, \cA^{\bop}\bt\cA)$-centered $(\cA\bt \cA, \cA\bt \cA^{\mop})$-bimodule.

\item $\epsilon_{\ev}$ is $\cA$ as an $(\HC(\cA), \un_\cS)$-bimodule.
\end{itemize}
\label{thm:2dualizabilityE2}
\end{theorem}

Next, we have the following analog of \cref{prop:E1Hom1adjoints} establishing dualizability for $2$-morphisms in $\alg_2(\cS)$.

\begin{proposition}[{\cite[Proposition 5.17]{Brochier2018}}]
\label{prop:forgetE2E1}
Let $\cC,\cD\colon\cA\to \cB$ be $1$-morphisms in $\alg_2(\cS)$, i.e. $(\cA,\cB)$-central algebras, and let $M\colon\cC\to\cD$ be a 2-morphism in $\alg_2(\cS)$, i.e. an $(\cA,\cB)$-central $(\cC,\cD)$-bimodule.  Then $M$ has a right (resp. left) adjoint in $\alg_2(\cS)$ if, and only if, it has a right (resp. left) adjoint in $\alg_1(\cS)$, when regarded as a $(\cC,\cD)$-bimodule.  The adjoints in $\alg_2(\cS)$ are given by equipping the adjoints in $\alg_1(\cS)$ with canonical central structures.
\end{proposition}

We recall the following analog of \cref{thm:easy2dualizability}, which reduces 3-dualizability to a finite list of conditions:
\begin{theorem}[{\cite[Proposition 1.2.1]{Araujo2017}}]\label{thm:easy3dualizability} Let $X$ be an object in a symmetric monoidal 3-category $\cC$. Suppose that $X$ has a dual and that the evaluation and coevaluation $1$-morphisms have right adjoints.  Then $X$ is 3-dualizable if, and only if, the unit and counit 2-morphisms witnessing each of these two adjunctions (four maps in total) have right adjoints.
\end{theorem}

We may finally prove a complete characterization of 3-dualizable objects in $\alg_2(\cS)$.

\begin{theorem}
An $E_2$-algebra $\cA\in\alg_2(\cS)$ is 3-dualizable if, and only if, $\cA$ is dualizable as an object of $\cS$, as an $\cA^e$-module, and as an $\HC(\cA)$-module.
\label{thm:3dualizabilityE2}
\end{theorem}
\begin{proof}
We recall from \cref{thm:2dualizabilityE2} the 2-dualizability data $\cA^\vee$, $\ev$, $\coev$, $\ev^R$, $\coev^R$, together with their units and counits $\eta_{\ev}, \eta_{\coev}, \epsilon_{\ev}, \epsilon_{\coev}$.  By \cref{thm:easy3dualizability}, $\cA$ is 3-dualizable if, and only if, these last four morphisms have right adjoints.  By \cref{prop:forgetE2E1} these exist if, and only if, the underlying bimodules have right adjoints which can be analyzed using \cref{prop:E1Hom1adjoints}:
\begin{itemize}
\item $\eta_{\ev}$ has a right adjoint if, and only if, $\cA$ is dualizable as a $\cA^e$-module.
\item $\eta_{\coev}$ has a right adjoint if, and only if, $\cA$ is dualizable as a $\HC(\cA)$-module.
\item $\epsilon_{\ev}$ has a right adjoint if, and only if, $\cA$ is dualizable as an object of $\cS$.
\item $\epsilon_{\coev}$ has a right adjoint if, and only if $\cA$ is dualizable as a $\cA\bt \cA$-module via the right $\cA\bt \cA$-action on $\cA$. Using the braiding we may identify $\cA\bt \cA\cong \cA\bt \cA^{\mop}$ as monoidal categories, so that under this identification the $\cA\bt \cA$-action on $\cA$ goes to the canonical $\cA^e$-action on $\cA$.
\end{itemize}
\end{proof}

\begin{remark}
Recall that the main result of \cite{Brochier2018} was the construction of a 3-dualizable subcategory of $\BrTens$ based on the notion of cp-rigidity, showing in particular that every cp-rigid braided tensor category $\cA$ is 3-dualizable.  As a byproduct, we gave sufficient condition for dualizability of higher morphisms, not just those appearing as dualizing data for $\cA$.  On the other hand, there was no proof that cp-rigidity was necessary for 3-dualizability of a braided tensor category, just that it was sufficient.  Indeed, since cp-rigidity is not {\it a priori} a Morita invariant, we do not expect it is a necessary condition.
  
By contrast, \cref{thm:Characterize3Dualizability} gives a complete characterization of 3-dualizability, it is Morita invariant, and the characterization holds for a general $\cS$.  However, in the case $\cS=\LFP$, let us underscore that it remains an open question to characterize necessary conditions for 1-dualizability in $\LFP$ \cite[Remark 3.6]{Brandenburg2015}, let alone as $\cA^e$- and $\HC(\cA)$-modules, so that in practice one must still appeal to the method of \cite{Brochier2018} to establish the conditions in \cref{thm:Characterize3Dualizability}.  \end{remark}

\subsection{Invertibility for $E_1$- and $E_2$-algebras}

The goal of this section is to give a complete characterization of invertible objects in $\alg_1(\cS)$ and $\alg_2(\cS)$. We begin with an elementary lemma.

\begin{lemma}
Suppose $\cC$ is a bicategory and $f\colon x\rightarrow y$ a 1-morphism. It is invertible if, and only if, it is right-adjointable and the unit and counit of the adjunction are isomorphisms.
\label{lm:invertibility}
\end{lemma}

By an iterated application of this lemma, we may give a straightforward characterization of invertible objects in $\alg_n(\cS)$. We begin with a characterization of invertible 1-morphisms in $\alg_2(\cS)$.

\begin{notation}
Give $\cC\in\alg_1(\cS)$, we denote $\cC^!=\Hom_{\cC^e}(\cC,\cC^e)$, and $\cC^*=\Hom_\cS(\cC,\un_\cS)$.
\end{notation}

\begin{theorem}
Suppose $\cB\in\alg_2(\cS)$ is an $E_2$-algebra and $\cC$ an $\cB$-central algebra viewed as a 1-morphism $\cB\rightarrow \un_\cS$. Then $\cC$ is invertible if, and only if, $\cC\in\alg_1(\cS)$ is 2-dualizable and the following maps are equivalences:
\begin{enumerate}
\item The evaluation map $\Hom_{\cC^e}(\cC, \cC^e)\bt_\cB \cC\rightarrow \cC^e$.

\item The map $\cB\rightarrow \rZ_1(\cC)$ given by the $\cB$-central structure on $\cC$.

\item The evaluation map $\Hom(\cC, \un_\cS)\bt_{\cC\bt_\cB \cC^{\mop}} \cC\rightarrow \un_\cS$.

\item The map $\cC^{\mop}\bt_\cB \cC\rightarrow \Hom(\cC, \cC)$ given by the left and right action of $\cC$ on itself.
\end{enumerate}
\label{thm:E2invertiblemorphism}
\end{theorem}
\begin{proof}
By \cref{thm:E2Hom1adjoints} $\cC$ admits a right adjoint $\cC^{\mop}\colon \un_\cS\rightarrow \cB$. The unit of the adjunction $\eta$ is $\cC$ viewed as an $(\cB, \cB)$-centered $(\cB, \cC\bt \cC^{\mop})$-bimodule. The counit of the adjunction $\epsilon$ is $\cC$ viewed as a $(\cC^{\mop}\bt_\cB \cC, \un_\cS)$-bimodule.

By \cref{lm:invertibility} the 1-morphism $\cC\colon \cB\rightarrow \un$ is invertible if, and only if, $\eta$ and $\epsilon$ are invertible. We will now analyze the invertibility of these 2-morphisms separately:
\begin{itemize}
\item By \cref{lm:invertibility} the unit $\eta$ is invertible if, and only if, it is right-adjointable with the unit and counit being isomorphisms. By \cref{prop:forgetE2E1,prop:E1Hom1adjoints} $\eta$ is right-adjointable if, and only if, $\cC$ is dualizable as a $\cC^e$-module, with dual $\cC^!$. The unit of this adjunction is $\cB\rightarrow \cC^!\bt_{\cC^e} \cC\cong \rZ_1(\cC)$. The counit of this adjunction is $\cC^!\bt_\cB \cC\rightarrow \cC^e$.

\item By \cref{lm:invertibility} the counit $\epsilon$ is invertible if, and only if, it is right-adjointable with the unit and counit being isomorphisms. By \cref{prop:E1Hom1adjoints} $\epsilon$ is right-adjointable if, and only if, $\cC$ is dualizable as an object of $\cS$, with dual $\cC^*$. The unit of this adjunction is $\cC^{\mop}\bt_\cB \cC\rightarrow \Hom(\cC,\un_\cS)\bt \cC\cong \Hom(\cC, \cC)$. The counit of this adjunction is $\cC^*\bt_{\cC\bt_\cB \cC^{\mop}} \cC\rightarrow \un_\cS$.
\end{itemize}
\end{proof}

The previous theorem recovers a well-known characterization of invertible objects in $\alg_1(\cS)$ (see \cite{MR0498794, MR1394505,  MR3258690} for related characterizations).

\begin{theorem}
An $E_1$-algebra $\cC\in \alg_1(\cS)$ is invertible if, and only if, it is 2-dualizable and the following maps are isomorphisms:
\begin{enumerate}
\item $\cC^{\mop}\bt \cC\rightarrow \Hom(\cC, \cC)$ given by the left and right action of $\cC$ on itself.

\item The inclusion of the unit $\un_\cS\rightarrow \rZ_1(\cC)$.
\end{enumerate}
\label{thm:E1invertible}
\end{theorem}
\begin{proof}
We have an equivalence $\alg_1(\cS)\cong \Hom_{\alg_2(\cS)}(\un_\cS, \un_\cS)$ of monoidal $(\infty,3)$-categories. Hence, $\cC\in\alg_1(\cS)$ is invertible if, and only if, it is invertible when viewed as a 1-morphism $\un_\cS\rightarrow \un_\cS$ in $\alg_2(\cS)$. By \cref{thm:E2invertiblemorphism} it is equivalent to $\cC\in\alg_1(\cS)$ being 2-dualizable and satisfying the 4 conditions of the theorem. Let us analyze them in pairs:
\begin{itemize}
\item The evaluation map $\cC^!\bt \cC\rightarrow \cC^e$ is a map of right $\cC^e$-modules, where on the lefthand side $\cC^e$ acts on $\cC^!$. Since $\cC\in\cS$ is dualizable and $\cC^!$ is dualizable as a $\cC^e$-module, $\cC^!\bt\cC$ is dualizable as a $\cC^e$-module. In particular, the evaluation map is an isomorphism if, and only if, its $\cC^e$-linear dual map is an isomorphism. But the dual map is $\cC^e\rightarrow \cC^*\bt \cC\cong \Hom(\cC, \cC)$, which is the map in the fourth condition of \cref{thm:E2invertiblemorphism}.

\item In the map $\un_\cS\rightarrow \rZ_1(\cC)\cong \cC^!\bt_{\cC^e} \cC$ both sides are dualizable in $\cS$: the dual of the righthand side is
\[\Hom(\cC^!\bt_{\cC^e} \cC, \un_\cS)\cong\Hom_{\cC^e}(\cC^!, \cC^*)\cong \cC^*\bt_{\cC^e} \cC.\]
In particular, the dual of this map is the map $\cC^*\bt_{\cC^e} \cC\rightarrow \un_\cS$ in the third condition of \cref{thm:E2invertiblemorphism}.
\end{itemize}
\end{proof}

\begin{remark}\label{rem:mates}
In the preceding proof an essential role was played by the dual $1$-morphisms.  Recall that the dual of a morphism $f\colon x \rightarrow y$ between dualizable objects is the composite
\[y^\vee \xrightarrow{\id\otimes\coev} y^\vee \otimes x \otimes x^\vee \xrightarrow{\id \otimes \phi \otimes \id} y^\vee \otimes y \otimes x^\vee \xrightarrow{\ev\otimes\id} x^\vee.\]
In fact, if we view modules as 1-morphisms in the $(\infty, 3)$-category $\alg_1(\cS)$, morphisms of modules are 2-morphisms. The dual module is the adjoint 1-morphism and the dual of a morphism of modules is called a ``mate'' of the 2-morphism.
\end{remark}

\begin{example}
Let $R$ be a commutative algebra and $\cS=\Mod_R$ the symmetric monoidal category of $R$-modules. \Cref{thm:E1invertible} characterizes invertible objects in $\alg_1(\Mod_R)$ as $R$-algebras $A$ satisfying the following four conditions:
\begin{enumerate}
\item $A$ is dualizable as an $R$-module (i.e., $A$ is finitely generated and projective as an $R$-module).
\item $A$ is dualizable as an $A^e$-module (i.e., it is separable).
\item The natural morphism $A^{\op}\otimes A\rightarrow \Hom_R(A, A)$ given by the left and right action is an isomorphism.
\item The morphism $R\rightarrow \rZ(A)$ is an isomorphism (i.e., $A$ is a central $R$-algebra).
\end{enumerate}
In this case one may prove a stronger claim that some of these conditions are equivalent to each other. Concretely, invertible objects in $\alg_1(\Mod_R)$ are $R$-algebras $A$ satisfying either of the following equivalent conditions (see \cite[Theorem II.3.4]{DeMeyerIngraham}):
\begin{itemize}
\item (\defterm{Azumaya}) $A$ is a faithful dualizable $R$-module and the morphism $A^{\op}\otimes A\rightarrow \Hom_R(A, A)$ is an isomorphism.

\item (\defterm{Central separable}) $A$ is a dualizable $A^e$-module and the morphism $R\rightarrow \rZ(A)$ is an isomorphism.
\end{itemize}
\label{ex:Azumaya}
\end{example}

Finally, we come to the main result of this section:

\begin{theorem}
An $E_2$-algebra $\cA\in\alg_2(\cS)$ is invertible if, and only if, it is 3-dualizable and the following maps are isomorphisms:
\begin{enumerate}
\item (cofactorizability) $\HC(\cA)\rightarrow \Hom(\cA, \cA)$.

\item (factorizability) $\cA\bt \cA^{\bop}\rightarrow \rZ_1(\cA)$.

\item (nondegeneracy) The inclusion of the unit $\un\rightarrow \rZ_2(\cA)$.
\end{enumerate}
\label{thm:E2invertible}
\end{theorem}
\begin{proof}
Clearly, an invertible $E_2$-algebra is 3-dualizable, so by \cref{thm:3dualizabilityE2} dualizability of $\cA$ as an $\HC(\cA)$-module is necessary. From now on we make this assumption. Recall that $\cA\in\alg_2(\cS)$ is 1-dualizable with the dual given by $\cA^{\bop}$. The evaluation map $\ev$ associated with this duality is $\cA$ viewed as a $\cA\bt \cA^{\bop}$-central algebra. Therefore, $\cA\in\alg_2(\cS)$ is invertible if, and only if, $\ev\colon \cA\bt \cA^{\bop}\rightarrow \un_\cS$ is an isomorphism. By \cref{thm:E2invertiblemorphism}, this is equivalent to the 4 conditions listed there.

Since $\cA$ is assumed to be dualizable over $\cA^e$, condition (1) is equivalent to the condition (2) after taking the dual over $\cA^e$. Condition (2) is precisely the condition that $\cA\bt \cA^{\bop}\rightarrow \rZ_1(\cA)$ is an isomorphism. Since $\cA$ is dualizable over $\HC(\cA)$, condition (3) is equivalent after applying $\Hom(-, \un_\cS)$ to the condition that $\un_\cS\rightarrow \Hom_{\HC(\cA)}(\cA, \cA)= \rZ_2(\cA)$ is an isomorphism. Finally, the condition (4) is the condition that the map $\HC(\cA)\rightarrow \Hom(\cA, \cA)$ is an isomorphism.
\end{proof}

\section{Invertibility of finite braided tensor categories}
\label{sec:tensor}
In this section we show that for finite braided tensor categories (in the sense of \cite{Etingof2004}), the three conditions for invertibility given in \cref{thm:intro-main-result-general} are already mutually equivalent.

\subsection{Categorical setup}
\label{sect:tensorcategories}

We begin by recalling some categorical background. All categories we consider are $k$-linear (i.e, enriched and tensored over $k$), where $k$ is an algebraically closed field of characteristic zero, and all functors are $k$-linear functors.

We will consider the symmetric monoidal 2-category $\LFP$ of locally presentable categories, their colimit preserving functors, and their natural transformations.  By the special adjoint functor theorem, a functor between locally presentable categories is colimit preserving if, and only if, it is a left adjoint: to emphasize this, we will use the notation $\Fun^L(\cC,\cD)$ in place of $\Hom_{\LFP}(\cC,\cD)$.  The symmetric monoidal structure is given by the so-called Deligne-Kelly tensor product.  For recollections about the notion of local presentability see \cite{Brochier2018}.  The most important class of locally presentable categories for us are those which have enough compact projectives.

\begin{definition} An object $X$ of a presentable category $\cC$ is compact-projective if the functor $\Hom(X,-)\colon\cC\to\Vect$ is colimit preserving.  The category $\cC$ \defterm{has enough compact projectives} if every object of $\cC$ can be expressed as a colimit of compact-projective objects.
\end{definition}

Equivalently, a category $\cC$ has enough compact projectives if, and only if, it is equivalent to the free cocompletion $\Fun(\tilde{\cC}^{\op}, \Vect)$ of a small category $\tilde{\cC}$, i.e. if it is a presheaf category.

\begin{remark}
It is shown in \cite{Brandenburg2015} that categories with enough compact-projectives are 1-dualizable as objects of $\LFP$, and conjectured there that these are the only 1-dualizable objects.
\end{remark}

\begin{definition}
A category $\cC$ \defterm{has a compact-projective generator} if the category $\tilde{\cC}$ may be taken to have a single object, and is \defterm{finite} if the endomorphism algebra of that object is finite-dimensional.
\end{definition}

Equivalently, a category with a compact-projective generator is one which is equivalent to $A$-mod for some associative algebra $A$, and a finite category is one for which $A$ may be taken to be finite-dimensional.  The following proposition gives a characterization of these notions internally to $\LFP$.  The proof is straightforward.

\begin{proposition}\label{prop:finitefromcompact}
Suppose that $\cC\in \LFP$ has enough compact projectives.
\begin{enumerate}
\item The identity endofunctor $\id_\cC\in\End_{\LFP}(\cC)$ is a compact object if, and only if, $\cC$ admits a compact-projective generator.
\item The identity endofunctor $\id_\cC\in\End_{\LFP}(\cC)$ is a projective object if, and only if, $\cC$ is semisimple.
\item Under assumption (1) (resp, (1) and (2)), $\cC$ is finite (resp, finite semisimple) if, and only if, Hom spaces between compact objects are finite-dimensional.
\end{enumerate}
\end{proposition}

\begin{remark}\label{rem:cat-2-dualizability} It follows easily from \cref{prop:finitefromcompact} that amongst categories with enough compact projectives the 2-dualizable categories are precisely the finite semisimple categories.  This also follows from \cite[Appendix A]{Bartlett2015} and \cite{Tillman1998}, because their category $\mathrm{Bim}$ (of Cauchy complete categories, bimodules, and bimodule maps) is equivalent -- via taking the free cocompletion of categories -- to the full subcategory of $\LFP$ whose objects are those with enough compact projectives. \end{remark}

We will use the term \defterm{tensor category} to mean an $E_1$-algebra in $\LFP$, and the term \defterm{braided tensor category} to mean an $E_2$-algebra in $\LFP$.  In particular, we will always assume that the underlying category in each case is locally presentable, and that the tensor product bifunctor $\cA\times\cA\rightarrow \cA$ preserves colimits in each variable, so that it defines a morphism $\cA\bt\cA\rightarrow\cA$ in $\LFP$. We introduce the notations $\Tens=\alg_1(\LFP)$ and $\BrTens=\alg_2(\LFP)$.

Let us begin by relating the Drinfeld and M\"uger centers to the more general notions introduced in \cref{sect:Enalgebras}. The following statement is proved in \cite[Proposition 7.13.8]{Etingof2015}.

\begin{proposition}
Let $\cC\in\alg_1(\LFP)$ be a tensor category. Then $\rZ_1(\cC)$ is equivalent to the Drinfeld center: the category of pairs $(x, \gamma)$, where $x\in\cC$ and $\gamma\colon (-)\otimes x\xrightarrow{\sim} x\otimes (-)$ is an associative natural isomorphism.
\label{prop:Drinfeldcenter}
\end{proposition}

We may also analyze the $E_2$-center of a braided tensor category.

\begin{proposition}
Let $\cA\in\alg_2(\LFP)$ be a braided tensor category. Then $\rZ_2(\cA)$ is equivalent to the M\"uger center: the full subcategory of $\cA$ consisting of objects $x\in\cA$ such that $\sigma_{y, x}\circ\sigma_{x, y}\colon x\otimes y\rightarrow x\otimes y$ is the identity for every $y\in\cA$.
\label{prop:Mugercenter}
\end{proposition}
\begin{proof}
Suppose $\cB\in\alg_2(\LFP)$ and $\cA$ is a $\cB$-central algebra. The central structure boils down to the data of a tensor functor $T\colon \cB\rightarrow \cA$ together with a natural isomorphism $\tau\colon T(z)\otimes x\rightarrow x\otimes T(z)$ for every $z\in\cB$ and $x\in\cA$. By \cite[Proposition 3.34]{Laugwitz2018} $\Hom_{\cA\bt_\cB \cA^{\mop}}(\cA, \cA)$ is a full subcategory of the Drinfeld center $\rZ_1(\cA)$ consisting of objects $(x, \gamma)$, where $\gamma_{T(z)}\colon T(z)\otimes x\rightarrow x\otimes T(z)$ coincides with $\tau$ for every $z\in\cB$.

The $E_2$-center $\rZ_2(\cA)$ is given by this construction with $\cB=\cA\bt\cA^{\bop}$. The $\cB$-central structure on $\cA$ sends $z\boxtimes\un\in\cB$ to $z\in\cA$ with $\tau$ given by $\sigma_{z, x}\colon z\otimes x\rightarrow x\otimes z$ and $\un\boxtimes z\in\cB$ to $z\in\cA$ with $\tau$ given by $\sigma_{x, z}^{-1}\colon z\otimes x\rightarrow x\otimes z$. Therefore, $\rZ_2(\cA)\subset \rZ_1(\cA)$ is a full subcategory consisting of objects $(x, \gamma)$, where
\[\gamma_z = \sigma_{x, z} = \sigma_{z, x}^{-1},\]
i.e. of objects lying in the M\"uger center.
\end{proof}

Let us recall some standard rigidity and finiteness assumptions on tensor categories.

\begin{definition}
Suppose a tensor category $\cA$ has enough compact projectives. 
\begin{itemize}
\item We say $\cA$ is \defterm{cp-rigid} if all compact projective objects of $\cA$ are dualizable.

\item We say $\cA$ is \defterm{compact-rigid} if all compact objects of $\cA$ are dualizable.

\item We say $\cA$ is a \defterm{finite tensor category} if it is compact-rigid, and its underlying category is finite.

\item We say $\cA$ is \defterm{fusion} if it is a finite tensor category and the underlying category is semisimple. 

\item We say $\cA$ is a \defterm{finite braided tensor category} (resp., \defterm{braided fusion category}) if it is braided, and its underlying tensor category is a finite tensor category (resp., fusion category).   
\end{itemize}
\end{definition}

These definitions are compatible with the most standard definition of \cite{Etingof2004} in the following sense: a tensor (resp., braided tensor) category is finite in the above sense if, and only if, it is the ind-completion of a finite tensor category (resp., finite braided tensor category) in the sense of \cite{Etingof2004} and \cite{ShimizuND}.  The next proposition is proved in \cite{Brochier2018}, by verifying closure under composition.\footnote{Both parts of this proposition require characteristic zero, otherwise one needs to restrict to fusion categories and braided fusion categories of nonzero global dimension.}

\begin{proposition}
We have higher subcategories of $\Tens$ and $\BrTens$, defined as follows:
\begin{itemize}
\item Fusion categories, semisimple bimodule categories, compact-preserving cocontinuous bimodule functors, and natural transformations form a subcategory $\Fus$ of $\Tens$.
\item  Braided fusion categories, fusion categories equipped with central structures, finite semisimple bimodule categories, compact-preserving cocontinuous bimodule functors, and bimodule natural transformations form a subcategory $\BrFus$ of $\BrTens$.
\end{itemize}
\end{proposition}

\begin{remark}
We warn the reader that although there is a $3$-category of finite tensor categories, finite bimodule categories, compact-preserving cocontinuous bimodule functors, and bimodule natural transformations which was the main object of study in \cite{Douglas2013}, we do not know of a similar $4$-category whose objects are finite braided tensor categories.  The issue is that the relative tensor product of finite tensor categories over a finite braided tensor category will not again be finite (because it will only be cp-rigid and not compact rigid).
\end{remark}

\begin{remark}\label{rem:BrTens-dualizability} It is shown in \cite{Douglas2013} and \cite{Brochier2018}, respectively, that finite tensor categories and cp-rigid tensor categories are 2-dualizable in $\Tens$.  It is shown in \cite{Brochier2018} that cp-rigid braided tensor categories are 3-dualizable in $\BrTens$.  Finally, it is shown in \cite{Douglas2013} and \cite{Brochier2018} that fusion categories are 3-dualizable in $\Tens$ and braided fusion categories are 4-dualizable in $\BrFus$.  We expect that a finite braided tensor category is 4-dualizable in $\BrTens$ if, and only if, its M\"uger center is semisimple, but we do not know a proof.
\end{remark}

\begin{remark}  Consider a non-semisimple and non-degenerate braided tensor category $\cA$, such as the category of representations for the small quantum group at a primitive $\ell$-th root of unity, where $\ell$ is odd, not divisible by the lacing number and coprime to the determinant of the Cartan matrix 
\cite{RossoQuantumGroups}, \cite[Chapter XI.6.3]{Turaev2010}, \cite{lentner_factorizable_2017}.  Note that this example is not braided fusion, and so its $4$-dualizability does not follow from \cite{Brochier2018}.  Furthermore, its underlying category of $\cA$ is not $2$-dualizable by \cref{rem:cat-2-dualizability}
However, simply because invertibility implies full dualizability, we may conclude in particular that $\cA$ is fully-dualizable.  In particular, we see that an $E_2$-algebra in $\cS$ can be $4$-dualizable even when the underlying object is not $2$-dualizable in $\cS$.
\end{remark}

We end this section with a result identifying $\HC(\cA)$ and $\rZ_1(\cA)$ as plain categories.

\begin{proposition}
Let $\cA$ be a cp-rigid braided tensor category. Then there is an equivalence of categories $\HC(\cA)\cong \rZ_1(\cA)$.
\end{proposition}
\begin{proof}
Consider the monoidal equivalence $L\colon \cA\rightarrow \cA^{\mop}$ which sends every compact projective object $x$ to the left dual ${}^\vee x$. By \cite[Proposition 3.13]{Ben-Zvi2009} and \cite[Theorem 3.2.4]{Douglas2013} we may identify
\[\rZ_1(\cA) = \Hom_{\cA^e}(\cA, \cA)\cong \cA\otimes_{\cA^e} {}_{\id} \cA_{LL},\]
where ${}_{\id} \cA_{LL}$ is the $(\cA, \cA)$-bimodule $\cA$ which has a regular left $\cA$-action, but whose right $\cA$-action is given by $LL$.

Now consider the identity functor $\id\colon \cA\rightarrow \cA$ equipped with the monoidal structure $\sigma^2\colon x\otimes y\xrightarrow{\sigma_{x, y}} y\otimes x\xrightarrow{\sigma_{y, x}} x\otimes y $. By \cite[Lemma 3.9]{Ben-Zvi2018a} we may identify
\[\HC(\cA)\cong \cA\otimes_{\cA^e} {}_{\id} \cA_{(\id, \sigma^2)},\]
where ${}_{\id} \cA_{(\id, \sigma^2)}$ is the $(\cA, \cA)$-bimodule $\cA$ which has a regular left $\cA$-action, but whose right $\cA$-action is given by the monoidal functor $(\id, \sigma^2)$.

But by \cite[Proposition 8.9.3]{Etingof2015} we have a natural monoidal isomorphism $(\id, \sigma^2)\Rightarrow LL$ which identifies the two bimodules.
\end{proof}

Note that in the cp-rigid case the natural monoidal structures on $\HC(\cA)$ and $\rZ_1(\cA)$ are nevertheless different as illustrated in \cref{ex:HCvsZ}. In the symmetric fusion case the compatibility between the two tensor structures on $\HC(\cA)\cong \rZ_1(\cA)$ is studied in \cite{Wasserman2020}.

From a TFT perspective, these monoidal structures may be understood as follows: the monoidal structure on $\HC(\cA)$ is obtained by embedding annuli inside one another (see \cite[Figure 1]{Ben-Zvi2018a}), while the monoidal structure on $\rZ_1(\cA)$ comes from the embedding of the two incoming and one outgoing annuli as the boundary of the pair of pants cobordism.  From this point of view it becomes clear that the latter tensor product is braided monoidal, while the former is only monoidal in general.  

\begin{example}\label{ex:HCvsZ}
Suppose $G$ is a finite group and let $\cA = \Rep(G)$ be the category of $G$-representations. Then $\HC(\cA)\cong \rZ_1(\cA)\cong\QCoh\left(\frac{G}{G}\right)$ is the category of adjoint-equivariant quasi-coherent sheaves on $G$. The symmetric tensor structure coming from $\HC(\cA)$ corresponds to the pointwise tensor product of quasi-coherent sheaves, while the braided tensor structure coming from $\rZ_1(\cA)$ corresponds to the convolution tensor structure.
\end{example}

\subsection{The canonical coend and end}

Let $\cA$ be a cp-rigid braided tensor category and denote by $\cA^{cp}\subset \cA$ the full subcategory of compact projective objects. The tensor product functor $T\colon \cA\bt \cA\rightarrow \cA$ admits a colimit-preserving right adjoint $T^R\colon \cA\rightarrow \cA\bt \cA$ (see e.g. \cite[Section 5.3]{Brochier2018}), so that we have a coend formula
\[T^R(\un_\cA) = \int^{x\in \cA^{cp}} x^\vee\boxtimes x\in\cA\bt \cA.\]

\begin{definition}
The \defterm{canonical coend} is the object $\cF\in\cA$ defined as
\[\cF = TT^R(\un_\cA) = \int^{x\in \cA^{cp}} x^\vee\otimes x.\]
\end{definition}

We denote by $\pi_x\colon x^\vee\otimes x\rightarrow \cF$ the natural projection. The canonical coend $\cF$ admits a natural structure of a braided Hopf algebra in $\cA$ illustrated in \cref{fig:coendHopfalgebra}.  These have been studied extensively, see e.g. \cite{Lyubashenko1994, Lyubashenko1995, Bruguieres2008, ShimizuND}.

\begin{figure}[h]
\begin{center}
\begin{tikzpicture}[thick]
  \draw[-{Latex[length=2mm]}] (0, 1) -- (0, 0.4);
  \draw (0, 0.4) -- (0, 0);
  \draw[-{Latex[length=2mm]}] (0.5, 0) -- (0.5, 0.6);
  \draw (0.5, 0.6) -- (0.5, 1);
  \draw[-{Latex[length=2mm]}] (1, 1) -- (1, 0.4);
  \draw (1, 0.4) -- (1, 0);
  \draw[-{Latex[length=2mm]}] (1.5, 0) -- (1.5, 0.6);
  \draw (1.5, 0.6) -- (1.5, 2);
  \draw (0, 1) .. controls (0, 1.6) and (0.5, 1.4) .. (0.5, 2);
  \draw (0.5, 1) .. controls (0.5, 1.6) and (1, 1.4) .. (1, 2);
  \draw (1, 1) .. controls (1, 1.2) .. (0.7, 1.35);
  \draw (0.55, 1.4) -- (0.35, 1.5);
  \draw (0.2, 1.6) .. controls (0, 1.8) .. (0, 2);
  \draw (0, -0.3) node {$x^\vee$};
  \draw (0.5, -0.35) node {$x$};
  \draw (1, -0.3) node {$y^\vee$};
  \draw (1.5, -0.35) node {$y$};

  \draw[-{Latex[length=2mm]}] (3, 0.6) -- (3, 0.4);
  \draw (3, 0.4) -- (3, 0);
  \draw (3, 0.6) .. controls (3, 1.4) .. (2.5, 2);
  \draw[-{Latex[length=2mm]}] (3.5, 0) -- (3.5, 0.6);
  \draw (3.5, 0.6) .. controls (3.5, 1.4) .. (4, 2);
  \draw (3, 2) arc (180:360:0.25);
  \draw (3, -0.3) node {$x^\vee$};
  \draw (3.5, -0.35) node {$x$};
  \draw (3, 2.2) node {$x$};
  \draw (3.5, 2.25) node {$x^\vee$};

  \draw[-{Latex[length=2mm]}] (4.5, 1) -- (4.5, 0.4);
  \draw (4.5, 0.4) -- (4.5, 0);
  \draw[-{Latex[length=2mm]}] (5, 0) -- (5, 0.6);
  \draw (5, 0.6) -- (5, 1);
  \draw (4.5, 1) arc (180:0:0.25);
  \draw (4.5, -0.3) node {$x^\vee$};
  \draw (5, -0.35) node {$x$};

  \draw[-{Latex[length=2mm]}] (7, 1) -- (7, 0.4);
  \draw (7, 0.4) -- (7, 0);
  \draw[-{Latex[length=2mm]}] (7.5, 0) -- (7.5, 0.6);
  \draw (7.5, 0.6) -- (7.5, 1);
  \draw (6, 1) arc (180:360:0.25);
  \draw (6.5, 1) .. controls (6.5, 1.6) and (7, 1.4) .. (7, 2);
  \draw (7, 1) .. controls (7, 1.6) and (7.5, 1.4) .. (7.5, 2);
  \draw (7.5, 1) .. controls (7.5, 1.3) .. (7.3, 1.4);
  \draw (7.1, 1.45) -- (6.9, 1.55);
  \draw (6.75, 1.6) .. controls (6.7, 1.7) and (6, 1.9) .. (6, 1.5);
  \draw (6, 1.5) -- (6, 1);
  \draw (6.9, 2.2) node {$x^{\vee\vee}$};
  \draw (7.6, 2.2) node {$x^\vee$};
  \draw (7, -0.3) node {$x^\vee$};
  \draw (7.5, -0.35) node {$x$};
\end{tikzpicture}
\end{center}
\caption{Multiplication $m$, coproduct $\Delta$, counit $\epsilon$ and antipode $S$ on $\cF$.}
\label{fig:coendHopfalgebra}
\end{figure}
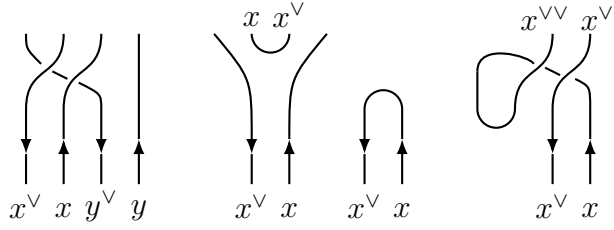

Moreover, $\cF$ is equipped with the following additional algebraic structures. There is a Hopf pairing $\omega\colon \cF\otimes \cF\rightarrow \un_\cA$ and an isomorphism $\tau_V\colon \cF\otimes V\rightarrow V\otimes \cF$ for every $V\in\cA$ illustrated in \cref{fig:Hopfpairing}. The isomorphism $\tau_V$ allows one to identify left $\cF$-modules with right $\cF$-modules, so the category $\Mod_\cF(\cA)$ inherits a monoidal structure given by the relative tensor product over $\cF$. We denote by
\[\triv_r\colon \cA\longrightarrow \Mod_\cF(\cA)\]
the functor which sends an object $V\in\cA$ to the trivial \emph{right} $\cF$-module. The following is proved in \cite[Section 4]{Ben-Zvi2018a} for compact-rigid categories, and can be extended to cp-rigid categories using \cite[Proposition 5.10]{Brochier2018}.

\begin{proposition}
We have an equivalence of monoidal categories $\HC(\cA)\cong \Mod_\cF(\cA)$. Under this equivalence the $\HC(\cA)$-module structure on $\cA$ is given by
\[M, V\mapsto M\otimes_\cF \triv_r(V).\]
\label{prop:HCREA}
\end{proposition}

\begin{figure}[h]
\begin{center}
\begin{tikzpicture}[thick]
  \draw[-{Latex[length=2mm]}] (0, 1) -- (0, 0.4);
  \draw (0, 0.4) -- (0, 0);
  \draw (0, 1) -- (0, 2);
  \draw[-{Latex[length=2mm]}] (0.5, 0) -- (0.5, 0.6);
  \draw (0.5, 0.6) -- (0.5, 1);
  \draw[-{Latex[length=2mm]}] (1, 1) -- (1, 0.4);
  \draw (1, 0.4) -- (1, 0);
  \draw[-{Latex[length=2mm]}] (1.5, 0) -- (1.5, 0.6);
  \draw (1.5, 0.6) -- (1.5, 2);
  \draw (0.5, 1) .. controls (0.5, 1.3) and (1, 1.2) .. (1, 1.5);
  \draw (1, 1) .. controls (1, 1.1) and (0.9, 1.15) .. (0.8, 1.22);
  \draw (0.7, 1.28) .. controls (0.6, 1.35) and (0.5, 1.42) .. (0.5, 1.5);
  \draw (0.5, 1.5) .. controls (0.5, 1.8) and (1, 1.7) .. (1, 2);
  \draw (1, 1.5) .. controls (1, 1.6) and (0.9, 1.65) .. (0.8, 1.72);
  \draw (0.7, 1.78) .. controls (0.6, 1.85) and (0.5, 1.92) .. (0.5, 2);
  \draw (0, 2) arc (180:0:0.25);
  \draw (1, 2) arc (180:0:0.25);
  \draw (0, -0.3) node {$x^\vee$};
  \draw (0.5, -0.35) node {$x$};
  \draw (1, -0.3) node {$y^\vee$};
  \draw (1.5, -0.35) node {$y$};

  \draw[-{Latex[length=2mm]}] (2.5, 1.5) -- (2.5, 0.4);
  \draw (2.5, 0.4) -- (2.5, 0);
  \draw[-{Latex[length=2mm]}] (3, 0) -- (3, 0.6);
  \draw (3, 0.6) -- (3, 1);
  \draw (3.5, 0) -- (3.5, 1);
  \draw (3, 1) .. controls (3, 1.3) and (3.5, 1.2) .. (3.5, 1.5);
  \draw (3.5, 1) .. controls (3.5, 1.1) and (3.4, 1.15) .. (3.3, 1.22);
  \draw (3.2, 1.28) .. controls (3.1, 1.35) and (3, 1.22) .. (3, 1.5);
  \draw (3, 1.5) .. controls (3, 1.8) and (2.5, 1.7) .. (2.5, 2);
  \draw (2.5, 1.5) .. controls (2.5, 1.6) and (2.6, 1.65) .. (2.7, 1.72);
  \draw (2.8, 1.78) .. controls (2.9, 1.85) and (3, 1.72) .. (3, 2);
  \draw (3.5, 1.5) -- (3.5, 2);
\end{tikzpicture}
\end{center}
\caption{Hopf self-pairing $\omega$ and the isomorphism $\tau_V$.}
\label{fig:Hopfpairing}
\end{figure}
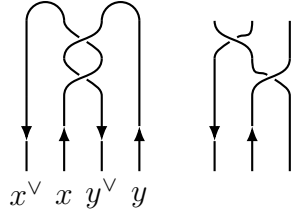

Dually, we may also consider the canonical end. Recall that $\Fun^L(\cA, \cA)$ denotes the category of colimit-preserving functors $\cA\rightarrow \cA$. Consider the tensor product functor $\tens\colon \cA\rightarrow \Fun^L(\cA, \cA)$ given by $x\mapsto x\otimes (-)$. It admits a right adjoint~\cite{Janelidze,ShimizuCoends}
\[\tens^R(F) = \int_{x\in\cA^{cp}}F(x)\otimes x^\vee.\]

\begin{definition}
The \defterm{canonical end} is the object $\cE\in\cA$ defined as
\[\cE = \tens^R(\id) = \int_{x\in \cA^{cp}} x\otimes x^\vee.\] 
\end{definition}

The object $\cE$ is naturally an algebra via the lax tensor structure on $\tens^R$.  We will use the following result.

\begin{proposition}
Let $\cA$ be a finite compact-rigid braided tensor category. Then
\[\tens^R\colon \Fun^L(\cA, \cA)\rightarrow \cA\]
is monadic and it identifies
\[\Fun^L(\cA, \cA)\cong \Mod_{\cE}(\cA).\]
\label{prop:tensormonadic}
\end{proposition}
\begin{proof}
Let $\cA^c\subset \cA$ be the full subcategory of compact objects. Since $\cA$ is locally finitely presentable, we may identify
\[\Fun^L(\cA, \cA)\cong \Ind\Fun^{rex}(\cA^c, \cA^c),\]
where $\Fun^{rex}(-, -)$ is the category of right exact functors.

Since $\cA$ is compact-rigid, $\cA^c$ is rigid. So, $\cA^c$ is an exact $\cA^c$-module category in the sense of \cite[Definition 3.1]{Etingof2004}. Clearly, it is also indecomposable. According to~\cite[Theorem 3.4]{ShimizuCoends}, $\tens^R$ restricts to an exact and faithful functor
\[\Fun^{rex}(\cA^c, \cA^c)\rightarrow \cA^c.\]
So, $\tens^R\colon \Fun(\cA, \cA)\rightarrow \cA$ is cocontinuous.

Since $\cA$ is compact-rigid, $\tens^R$ carries an $\cA$-module structure, so that the composition $\tens^R\circ \tens$ is canonically isomorphic to the endofunctor of $\cA$ given by tensoring with $\cE$. The result then follows from the standard monadic argument (see~\cite[Section~4.1]{Ben-Zvi2018a}).
\end{proof}

\subsection{Cofactorizability and the Drinfeld map}

There is a canonical \defterm{Drinfeld map} $\Dr\colon \cF\rightarrow \cE$ shown in \cref{fig:Drinfeldmap}.

\begin{figure}[h]
\begin{center}
\begin{tikzpicture}[thick]
  \draw[-{Latex[length=2mm]}] (0, 1) -- (0, 0.4);
  \draw (0, 0.4) -- (0, 0);
  \draw[-{Latex[length=2mm]}] (1, 0) -- (1, 0.6);
  \draw (1, 0.6) -- (1, 1);
  \draw (0, 1) arc (180:160:0.5);
  \draw (1, 1) arc (0:140:0.5);
  \draw[-{Latex[length=2mm]}] (0, 1.5) -- (0, 2.1);
  \draw (0, 2.1) -- (0, 2.5);
  \draw[-{Latex[length=2mm]}] (1, 2.5) -- (1, 1.9);
  \draw (1, 1.9) -- (1, 1.5);
  \draw (0, 1.5) arc (180:320:0.5);
  \draw (1, 1.5) arc (0:-20:0.5);
  \draw (0, -0.3) node {$x^\vee$};
  \draw (1, -0.35) node {$x$};
  \draw (0, 2.8) node {$y$};
  \draw (1, 2.85) node {$y^\vee$};
\end{tikzpicture}
\end{center}
\caption{Drinfeld map $\Dr\colon\cF\rightarrow \cE$.}
\label{fig:Drinfeldmap}
\end{figure}

In this section we establish that in the finite setting cofactorizability is equivalent to invertibility of the Drinfeld map.

\begin{proposition}
Let $\cA$ be a finite compact-rigid braided tensor category. It is cofactorizable if, and only if, the Drinfeld map $\Dr\colon \cF\rightarrow \cE$ is an isomorphism.
\label{prop:Drinfeldcofactorizable}
\end{proposition}
\begin{proof}
Let $\free\colon \cA\rightarrow \HC(\cA)$ be the functor $x\mapsto x\otimes \cF$ sending $x\in\cA$ to the free right $\cF$-module; its right adjoint is the forgetful functor $\forget\colon \HC(\cA)\rightarrow \cA$.

Consider the commutative diagram
\[
\xymatrix{
\HC(\cA) \ar[rr] && \Hom(\cA, \cA) \\
& \cA \ar^{\free}[ul] \ar_{\tens}[ur]
}
\]

Passing to right adjoints of vertical functors by \cref{prop:HCREA} and \cref{prop:tensormonadic} we get monadic functors. Therefore, the functor $\HC(\cA)\rightarrow \Hom(\cA, \cA)$ is an equivalence if, and only if, the associated functor of monads $\forget\circ\free\Rightarrow \tens^R\circ\tens$ is an equivalence. Since both monads are given by tensor product with an algebra, it is enough to show that the value of the above functor on $\un_\cA$ is an isomorphism, i.e. that the map
\[(\forget\circ\free)(\un_\cA) = \cF\longrightarrow (\tens^R\circ\tens)(\un_\cA) = \cE\]
is an isomorphism.

For $V\in\cA$ we have a commutative diagram
\[
\xymatrix{
\Hom_{\cA}(V, \cF) \ar[r] \ar^{\sim}[d] & \Hom_{\cA}(V, \cE) \ar^{\sim}[d] \\
\Hom_{\HC(\cA)}(V\otimes \cF, \cF) \ar[r] & \Hom_{\Hom(\cA, \cA)}(V\otimes (-), \id)
}
\]
where the map at the bottom sends $f\colon V\otimes \cF\rightarrow \cF$ to the bottom map in the commutative diagram
\[
\xymatrix{
(V\otimes \cF)\otimes_{\cF} y \ar^-{f\otimes \id}[r] \ar^{\sim}[d] & \cF\otimes_{\cF} y \ar^{\sim}[d] \\
V\otimes y \ar[r] & y
}
\]
Taking $V=\cF$ equipped with the identity map $\cF\rightarrow \cF$ the induced map $\cF\otimes y\rightarrow y$ is given by treating $y\in\cA$ as a left $\cF$-module via $\triv_r$. The map $\cF\rightarrow y\otimes y^\vee$ is therefore given by the $y\otimes y^\vee$-component of the Drinfeld map $\Dr\colon \cF\rightarrow \cE$, see \cref{fig:Rossoaction}.

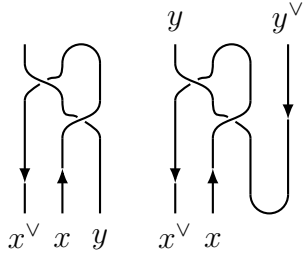
\begin{figure}[h]
\begin{center}
\begin{tikzpicture}[thick]
  \draw[-{Latex[length=2mm]}] (0, 1.5) -- (0, 0.4);
  \draw (0, 0.4) -- (0, 0);
  \draw[-{Latex[length=2mm]}] (0.5, 0) -- (0.5, 0.6);
  \draw (0.5, 0.6) -- (0.5, 1);
  \draw (1, 0) -- (1, 1);
  \draw (0.5, 1) .. controls (0.5, 1.3) and (1, 1.2) .. (1, 1.5);
  \draw (1, 1) .. controls (1, 1.1) and (0.9, 1.15) .. (0.8, 1.22);
  \draw (0.7, 1.28) .. controls (0.6, 1.35) and (0.5, 1.22) .. (0.5, 1.5);
  \draw (0.5, 1.5) .. controls (0.5, 1.8) and (0, 1.7) .. (0, 2);
  \draw (0, 1.5) .. controls (0, 1.6) and (0.1, 1.65) .. (0.2, 1.72);
  \draw (0.3, 1.78) .. controls (0.4, 1.85) and (0.5, 1.72) .. (0.5, 2);
  \draw (0, 2) -- (0, 2.25);
  \draw (1, 1.5) -- (1, 2);
  \draw (1, 2) arc (0:180:0.25);
  \draw (0, -0.3) node {$x^\vee$};
  \draw (0.5, -0.35) node {$x$};
  \draw (1, -0.35) node {$y$};

  \draw[-{Latex[length=2mm]}] (2, 1.5) -- (2, 0.4);
  \draw (2, 0.4) -- (2, 0);
  \draw[-{Latex[length=2mm]}] (2.5, 0) -- (2.5, 0.6);
  \draw (2.5, 0.6) -- (2.5, 1);
  \draw (3, 0.25) -- (3, 1);
  \draw (2.5, 1) .. controls (2.5, 1.3) and (3, 1.2) .. (3, 1.5);
  \draw (3, 1) .. controls (3, 1.1) and (2.9, 1.15) .. (2.8, 1.22);
  \draw (2.7, 1.28) .. controls (2.6, 1.35) and (2.5, 1.22) .. (2.5, 1.5);
  \draw (2.5, 1.5) .. controls (2.5, 1.8) and (2, 1.7) .. (2, 2);
  \draw (2, 1.5) .. controls (2, 1.6) and (2.1, 1.65) .. (2.2, 1.72);
  \draw (2.3, 1.78) .. controls (2.4, 1.85) and (2.5, 1.72) .. (2.5, 2);
  \draw (2, 2) -- (2, 2.25);
  \draw (3, 1.5) -- (3, 2);
  \draw (3, 2) arc (0:180:0.25);
  \draw (3, 0.25) arc (180:360:0.25);
  \draw[-{Latex[length=2mm]}] (3.5, 2.25) -- (3.5, 1.25);
  \draw (3.5, 1.25) -- (3.5, 0.25);
  \draw (2, -0.3) node {$x^\vee$};
  \draw (2.5, -0.35) node {$x$};
  \draw (2, 2.55) node {$y$};
  \draw (3.5, 2.6) node {$y^\vee$};
\end{tikzpicture}
\end{center}
\caption{The action map $\cF\otimes \triv_r(y)\rightarrow \triv_r(y)$ and the corresponding map $\cF\rightarrow y\otimes y^\vee$.}
\label{fig:Rossoaction}
\end{figure}
\end{proof}

We can now collect all results about invertibility of finite braided tensor categories in the following statement.

\begin{theorem}\label{thm:body-main-result}
Let $\cA$ be a compact-rigid finite braided tensor category. The following conditions are equivalent:
\begin{enumerate}
\item $\cA$ is invertible.

\item $\cA$ is \defterm{non-degenerate}: the natural functor $\Vect\rightarrow \rZ_2(\cA)$ is an equivalence.

\item $\cA$ is \defterm{factorizable}: the natural functor $\cA\bt \cA^{\bop}\rightarrow \rZ_1(\cA)$ is an equivalence.

\item $\cA$ is \defterm{cofactorizable}: the natural functor $\HC(\cA)\rightarrow \Hom(\cA, \cA)$ is an equivalence.

\item The Hopf pairing $\omega\colon \cF\otimes \cF\rightarrow \un_\cA$ is non-degenerate.

\item The Drinfeld map $\Dr\colon\cF\rightarrow \cE$ is an isomorphism.
\end{enumerate}
\end{theorem}
\begin{proof}
\cite[Theorem 1.1]{ShimizuND} establishes an equivalence between conditions $(2)$, $(3)$ and $(5)$. \cite[Proposition 4.11]{Farsad2018} establishes an equivalence between conditions $(5)$ and $(6)$. \Cref{prop:Drinfeldcofactorizable} establishes an equivalence between conditions $(4)$ and $(6)$.

Finally, \cref{thm:E2invertible} asserts that condition $(1)$ is equivalent to a combination of conditions $(2)$, $(3)$ and $(4)$ which finishes the proof.
\end{proof}

\section{The Picard group of $\BrTens$ and the Witt group}
\label{sec:Witt}
In this section we recall the Witt group of non-degenerate braided fusion categories \cite{MR3039775}, and identify it with the Picard group of $\BrFus$.  We then state a number of questions concerning the Picard group of $\BrTens$, which we regard as a non-fusion generalization of the Witt group.

Recall that non-degenerate braided fusion categories form a monoid under the Deligne-Kelly tensor product.  Non-degenerate braided fusion categories $\cA$ and $\cB$ are \defterm{Witt equivalent} if there exist fusion categories $\cC$ and $\cD$ with $\cA\bt \rZ_1(\cC) \simeq \cB\bt \rZ_1(\cD)$ as braided tensor categories.  That is, Witt equivalence is the equivalence relation generated by equating categories which are braided tensor equivalent, and setting Drinfeld centers to be trivial.

\begin{definition}[\cite{MR3039775}]
The \defterm{Witt group of non-degenerate braided fusion categories} is the quotient monoid of non-degenerate braided fusion categories, by Witt equivalence.  The inverse operation is $[\cA]^{-1} = [\cA^{\bop}]$, due to the factorizability property $\rZ_1(\cA)\simeq \cA\bt\cA^{\bop}$ of non-degenerate braided fusion categories.
\end{definition}


Recall that the Picard group $\Pic(\cT)$ of a symmetric monoidal $n$-category $\cT$ is the group whose elements are equivalence classes of invertible objects and whose composition is given by tensor product.  The main results of this paper give a concrete description of the elements of the Picard group of $\BrTens$ and its subcategory $\BrFus$.  In the latter case, we have:

\begin{theorem}\label{thm:WittBrFus}
The Picard group of $\BrFus$ is naturally isomorphic to the Witt group of non-degenerate braided fusion categories.
\end{theorem}
\begin{proof}
By \cref{thm:body-main-result}, the non-degenerate braided fusion categories are exactly the invertible objects of $\BrFus$, so it only remains to show that the a $\cA$-central fusion category $\cC$ gives an equivalence between $\cA$ and $\Vect$ if and only the natural map $\cA \rightarrow \rZ_1(\cC)$ is an equivalence.  This was already proved in \cite[Theorem 2.23]{Jones}.  

We also give an alternate proof using our techniques. Let us apply \cref{thm:E2invertiblemorphism} in the case $\cS=\LFP$ and $\cB=\cA$. Note that $\cC$ is dualizable as an $\cA$- and as a $\cC^{\mop} \boxtimes_\cA \cC$-module (e.g. by \cite[Theorem 5.16]{Brochier2018}). Then we obtain that an $\cA$-central fusion category $\cC$ gives an equivalence between $\cA$ and $\Vect$ if and only if the following four functors define equivalences of categories:
\begin{enumerate}
\item $\cC^e \rightarrow \Hom_\cA(\cC,\cC)$ given by the left and right action of $\cC$ on itself.
\item $\cA \rightarrow \rZ_1(\cC)$ given by the $\cA$-central structure on $\cC$.
\item $\Vect \rightarrow \Hom_{\cC^{\mop} \boxtimes_\cA \cC}(\cC,\cC)$ given by the inclusion of the identity.
\item $\cC^{\mop} \boxtimes_\cA \cC \rightarrow \Hom(\cC,\cC)$ given by the left and right action of $\cC$ on itself.
\end{enumerate}

Condition (2) holds by assumption.  Using (2), condition (3) reduces to the triviality of $\Hom_{\cC^{\mop} \boxtimes_{\rZ_1(\cC)} \cC}(\cC,\cC)$.  As in the proof of Proposition \ref{prop:Mugercenter}, we use \cite[Proposition 3.34]{Laugwitz2018} to rewrite $\Hom_{\cC^{\mop} \boxtimes_{\rZ_1(\cC)} \cC}(\cC,\cC)$ as the full subcategory of $\rZ_1(\cC)$ consisting of objects $(x, \gamma)$, where for every other pair $(y, \gamma')$ we have $\gamma_y = ({\gamma'}_x)^{-1}$ as maps $x \otimes y \rightarrow y \otimes x$.  Equivalently, ${\gamma'}_x \circ \gamma_y = \id_{x \otimes y}$, i.e. $(x, \gamma)$ lies in the M\"uger center.  Thus, condition (3) becomes the triviality of $\rZ_2(\rZ_1(\cC))$, which automatically holds for finite tensor categories, see \cite[Proposition 4.4]{Etingof2004a}.

We then have that (2) implies (1) and (3) implies (4) by the double commutant theorem \cite[Theorem 7.12.11]{Etingof2015}.
\end{proof}

According to \cref{thm:WittBrFus}, we may regard $\Pic(\BrTens)$ as a natural generalization of the Witt group, without the finite and semisimple assumptions.  The inclusion of $\BrFus$ into $\BrTens$ induces a group homomorphism 
\[\rho:\Pic(\BrFus)\to\Pic(\BrTens).\]  This observation leads to a number of interesting and apparently non-trivial questions.  Let us stress that we are not venturing conjectural answers to any of these questions.

\begin{question}
Is the homomorphism $\rho$ injective?  In other words, can it happen that two non-degenerate braided fusion categories are equivalent in $\BrTens$, via a central algebra which is not itself fusion, hence not a $1$-morphism in $\BrFus$?\end{question}

\begin{question} If a finite tensor category is trivial in the Witt group, must it be the center of a \emph{finite} tensor category?
\end{question}
\begin{question} Is the homomorphism $\rho$ surjective?  In other words, is every invertible braided tensor category in fact equivalent in $\BrTens$ to a non-degenerate braided fusion category?\end{question}
\begin{question}\label{q:invimplfin} Is every invertible braided tensor category equivalent in $\BrTens$ to a \emph{finite} braided tensor category?
\end{question} 

\begin{question} Is the Drinfeld center of any finite tensor category trivial in $\Pic(\BrTens)$?
We note that Drinfeld centers of infinite tensor categories are not typically invertible, let alone trivial in $\Pic(\BrTens)$.  We note also that condition (2) in \cref{thm:E2invertiblemorphism} establishes the reverse implication, that trivial elements of $\Pic(\BrTens)$ necessarily represent Drinfeld centers.  However, the argument from \cref{thm:WittBrFus} showing that centers of fusion categories are trivial does not apply, {\it a priori}.  This is because the relative tensor product of finite tensor categories over a centrally acting braided tensor category might not be finite (in particular, might not be compact-rigid) and so the double-commutant theorem does not apply.  \end{question}

Recall that the higher Picard groupoid $\hPic(\cT)$ of a symmetric monoidal $n$-category is the subgroupoid of invertible objects and invertible higher morphisms in $\cT$.  By definition, we have $\pi_0(\hPic(\cT))=\Pic(\cT)$ but it is then interesting to study higher homotopy groups.  It is known that $\pi_1$, $\pi_2$, and $\pi_3$ of $\hPic(\BrFus)$ vanish (the proof of this uses the notion of FP-dimensions), and that $\pi_4 = k^\times$.
\begin{question} What is the Postnikov k-invariant relating the $\pi_0$ and $\pi_4$ of $\hPic(\BrFus)$?
\end{question}

\begin{question} What is the homotopy type of $\hPic(\BrTens)$?
\end{question}

A symmetric tensor category may be regarded as an $E_k$-algebra in $\LFP$ for any $k\geq 3$. Moreover, as observed in \cite[Section 1.2]{Haugseng2017}, one has $\Omega\hPic(\alg_k(\LFP))\cong \hPic(\alg_{k-1}(\LFP))$. So, the collection of symmetric monoidal $\infty$-groupoids $\hPic(\alg_k(\LFP))$ forms a spectrum.

\begin{question}
What are the homotopy groups of this spectrum?
\end{question}




\bibliographystyle{style}
\bibliography{biblio}
\end{document}